\documentclass[ reqno]{amsart}
\usepackage{amsthm}
\usepackage{url}
\usepackage{tikz}
\usetikzlibrary[matrix, arrows, arrows.meta, calc, decorations.pathmorphing, backgrounds, positioning, fit, petri]
\usepackage{amscd}   
\usepackage[all, graph, poly]{xy} 
\usepackage{amssymb, amsmath}
\usepackage{multicol}

\newcommand{\C}{\mathbb{C}}
\newcommand{\F}{\mathbb{F}}
\newcommand{\fg}{\mathfrak{g}}

\newcommand{\OO}{\mathbb{O}}

\newcommand{\Z}{\mathbb{Z}}

\newcommand{\msf}[1]{\mathsf{#1}}

\newcommand{\op}[1]{\operatorname{\msf{#1}}}
\swapnumbers
\newtheorem{theorem}{Theorem}
\newtheorem{lemma}[theorem]{Lemma}

\theoremstyle{definition}
\newtheorem{definition}[theorem]{Definition}

\newtheorem{topic}[theorem]{}
\theoremstyle{remark}
\newtheorem{remark}[theorem]{Remark}
\begin{document}
%
%
%
%
\title[]{Root space decomposition of $\mathfrak{g}_2$ from octonions}
\author{Tathagata Basak}
\address{Department of Mathematics\\Iowa State University \\Ames, IA 50011}
\email{tathagat@iastate.edu}
\urladdr{http://orion.math.iastate.edu/tathagat}
\date{July 20, 2017}
\begin{abstract}
We describe a simple way to  write down explicit derivations of octonions
that form a Chevalley basis of $\mathfrak{g}_2$. This uses the description of 
octonions as a twisted group algebra of the finite field $\mathbb{F}_8$.
Generators of $\op{Gal}(\mathbb{F}_8/\mathbb{F}_2)$ act on the roots
as $120$--degree rotations and complex conjugation acts
as negation.
\end{abstract}
\maketitle
%
%
\begin{topic}{\bf Introduction.} 
Let $\OO$ be the unique real nonassociative eight dimensional division algebra of octonions.
It is well known that the Lie algebra of derivations $\op{Der}(\OO)$ 
is the compact real form of the Lie algebra of type $G_2$.
Complexifying we get an identification of $\op{Der}(\OO) \otimes \C$
with the complex simple Lie algebra $\mathfrak{g}_2$. 
 The purpose of this short note is to make this identification transparent by
writing down simple formulas for a set of 
derivations of $\OO \otimes \C$ that form a Chevalley
basis of $\mathfrak{g}_2$ (see Theorem \ref{th-Chevalley-basis}). 
This gives a quick construction of 
$\mathfrak{g}_2$ acting on $\op{Im}(\OO) \otimes \C$ because
the root space decomposition is visible from the definition.
The highest weight vectors of finite dimensional irreducible representations of $\mathfrak{g}_2$
can also be  easily described in these terms.
For an alternative construction
of $\mathfrak{g}_2$ as derivations of split octonions, see \cite{KT:OD}, pp. 104--106.
\par
Wilson \cite{W:CR} gives an elementary construction of the compact real form of $\mathfrak{g}_2$
with visible $2^3 \cdot L_3(2)$ symmetry.
This note started as a reworking of that paper in light of the definition of $\OO$
from \cite{B:O}, namely, that $\OO$ can be defined as the real algebra with basis
$\lbrace e^x \colon x \in \F_8 \rbrace$,
with multiplication defined by 
\begin{equation}
e^x e^y = (-1)^{\varphi(x, y)} e^{x + y} \text{\; where \;} \varphi(x, y) = \op{tr}(y x^6)
\label{eq-definition-of-O}
\end{equation}
and $\op{tr}: \F_8 \to \F_2$ is the trace map: $x \mapsto x + x^2 + x^4$. 
\par
The definition of $\OO$
given above has a visible order--three symmetry $\op{Fr}$ corresponding to the Frobenius automorphism 
$x \mapsto x^2$ generating $\op{Gal}(\F_8/\F_2)$, and
a visible order--seven symmetry $\op{M}$ corresponding to multiplication by a generator of $\F_8^*$. 
Together they generate a group of order $21$ that acts simply transitively on 
the natural basis $B = \lbrace e^x \wedge e^y \colon x, y \in \F_8^*, x \neq y \rbrace$ of
$\wedge^2 \op{Im}(\OO)$.
The only element of $\F_8^*$ fixed by $\op{Fr}$ is $1$. Let $\lbrace 0, 1, x, y \rbrace \subseteq \F_8$ be a subset corresponding to any line of 
$P^2(\F_2)$ containing $1$. Let $B_0 \subseteq B$ be the Frobenius orbit of $e^x \wedge e^y$ and 
let $B_0, B_1 \dotsb, B_6$ be the seven translates of $B_0$ by the cyclic group $\langle \op{M} \rangle$.
Then $B$ is the disjoint union of $B_0, \dotsb B_6$.
\par
Using the well known natural surjection $D: \wedge^2 \op{Im}(\OO) \to \op{Der}(\OO)$,
we get a generating set $D(B)$ of $\op{Der}(\OO) $.
The kernel of $D$ has dimension seven with a basis
$ \lbrace \sum_{b \in B_i} b \colon i = 0, \dotsb 6 \rbrace$. 
The images of $B_0, \dotsb, B_6$ span seven mutually orthogonal Cartan subalgebras transitively
permuted by $\langle \op{M} \rangle$ and forming an orthogonal decomposition of $\op{Der}(\OO)$
in the terminology of \cite{KT:OD}.
We fix the Cartan subalgebra spanned by $D(B_0)$ because it is stable under the action of $\op{Fr}$.
The short coroots in this Cartan are $\lbrace \pm D(b) \colon b \in B_0 \rbrace$.
At this point, it is easy to write down explicit derivations of $\OO \otimes \C$ corresponding
to a Chevalley basis of $\mathfrak{g}_2$ by simultaneously diagonalizing the
action of the coroots (see the discussion preceding
Theorem \ref{th-Chevalley-basis}).
\par
The following symmetry considerations make our job easy.
The reflections in the short roots of $\mathfrak{g}_2$
generate an $S_3$ that has index $2$ in the Weyl group.
In our description, the action of this $S_3$ is a-priori visible. 
This $S_3$ is generated by $\op{Fr}$ acting as
$120$--degree rotation and
the complex conjugation on $\OO \otimes \C$ acting as negation.
 \end{topic}
\begin{definition} 
Let $a, b \in \OO$. Write $\op{ad}_a(b) =[a, b] = a b  - b a$.
Define $D(a, b): \OO \to \OO$ by
\begin{equation*}
D(a, b) = \tfrac{1}{4}( [\op{ad}_{a}, \op{ad}_b] + \op{ad}_{[a, b]}).
\end{equation*}
Clearly $D(a, b) =  - D(b, a)$, $D(a, b) 1 = 0$ and $D(1, a) = 0$.
So $D$ defines a linear map from $\wedge^2 \op{Im}(\OO)$ to $\op{End}(\OO)$
which we also denote by $D$. So $D(a, b) = D(a \wedge b)$.
\end{definition}
 {\bf Notation:} From here on, we shall write $\mathfrak{g} = \op{Der}(\OO) \otimes \C$.
 The Frobenius automorphism $\op{Fr}$
 acts on $\OO$, on $\mathfrak{g}$, on the roots of $\mathfrak{g}$, and so on.
 If $x$ is an element of any of these sets, we sometimes write $x'$ for its image under $\op{Fr}$.
 Choose $\alpha \in \F_8$ such that $\alpha^3 = \alpha + 1$. 
Write
\begin{equation*}
e_i = e^{\alpha^i} \text{\; and \;} e_{i j} = D(e_i \wedge e_j).
\end{equation*}
Note that $\op{Fr}: e_i \mapsto e_{2 i}$, that is, $e_i' = e_{2i}$,
 where the subscripts are read modulo $7$.
\begin{lemma}
Let $x$ and $y$ be distinct elements of  $\F_8^*$ and $z \in \F_8$. Then 
\begin{equation*}
D(e^x \wedge e^y) e^z =  
\begin{cases}
2e^y &\text{\; if \;} z = x, \\
-2e^x & \text{\; if \;} z = y, \\
0 & \text{\; if \;} z = 0 \text{\; or \;} z = x + y, \\
-(e^x e^y) e^z  &\text{\; otherwise}.
\end{cases} 
\end{equation*}
\label{l-formula-for-D}
\end{lemma}
\begin{proof}Let $a, b \in \OO$. Define $R(a, b): \OO \to \OO$ by
$R(a, b)= [\op{ad}_{a}, \op{ad}_b] - \op{ad}_{[a, b]} $.
One verifies that $ R(a_1,a_2)(a_3) 
= - \sum_{\sigma \in S_3} \op{sign}(\sigma) [a_{\sigma(1)}, a_{\sigma(2)}, a_{\sigma(3)} ]$
where $[a, b, c] = (a b) c - a (b c)$ is the associator.
The properties of the associator in $\OO$ implies
$R(a_1,a_2)(a_3) = -6[a_1, a_2, a_3]$.
So
\begin{equation}
2 D(a \wedge b) =  \op{ad}_{[a, b]} +  \tfrac{1}{2} R(a, b) = \op{ad}_{[a, b] } - 3 [a, b, \cdot].
\label{e-second-formula-for-D}
\end{equation}
If $z \in \mathbb{F}_2 x + \mathbb{F}_2 y$, then $e^z$ belongs to the associative subalgebra spanned by $e^x$ and $e^y$
and the Lemma is easily verified in this case.  
If $z \notin \mathbb{F}_2 x +\mathbb{F}_2 y$, then using equation  \eqref{eq-definition-of-O} one easily verifies that
$\op{ad}_{[e^x, e^y]} e^z = 2[e^x, e^y, e^z]$. 
 The Lemma follows from this and equation \eqref{e-second-formula-for-D}.
\end{proof}
%
%
Let $a, b \in \OO$. Since the subalgebra of $\OO$ generated by $a$ and $b$ is associative, 
the maps $\op{ad}_{[a, b]} $ and $[\op{ad}_a, \op{ad}_b]$ agree on this subalgebra.
Note that the restriction of $2 D(a, b)$ to this subalgebra is just the inner derivation
 $\op{ad}_{[a, b]}$. In fact the following is well known: 
\begin{lemma}
If $a, b \in \OO$, then $D(a \wedge b)$ is a derivation of $\OO$.
\label{l-der}
\end{lemma}
By linearity it suffices to show that if $x, y$ are distinct elements of $\F_8^*$, then
$\mathfrak{d} = D(e^x \wedge e^y)$ is a derivation of $\OO$.
Write 
$L(z, w) = \mathfrak{d}(e^z) e^w + e^z \mathfrak{d}(e^w) - \mathfrak{d}(e^z e^w)$.
It suffices to prove that $L(z, w) = 0$ for all $z, w \in \F_8$. 
Only a few cases need to be checked if one first proves the following Lemma.
\begin{lemma}
(a) Suppose $u + x + y$ and $v$ are distinct elements of $\F_8^*$. 
Then $\mathfrak{d}(e^u)$ and $e^v$ anticommute.
\par 
(b) Suppose $u, v$ and  $x+y$ are three distinct elements of $\F_8^*$.
If $L(u, v) = 0$, then $L(u + v , u) = 0$.
\label{l-LL}
\end{lemma}
One can directly prove Lemmas \ref{l-der} and \ref{l-LL} using Lemma \ref{l-formula-for-D}.
Since Lemma \ref{l-der} is well known (see \cite{S:N}), we shall omit the details of the proof 
and move on to describe the kernel of $D: \wedge^2 \op{Im}(\OO) \to \op{Der}(\OO)$.
Let $\op{M} : \F_8 \to \F_8$ be the automorphism $\op{M}(x) = \alpha x$. 
Let $\tau = \op{M}$ or $\tau = \op{Fr}$. 
Recall the multiplication rule of $\OO$ from equation \eqref{eq-definition-of-O}.
Note that $\varphi(\tau x, \tau y) = \varphi(x, y)$. 
It follows that $(a b)^{\tau} = a^{\tau} b^{\tau} $ for $a, b \in \OO$ where
$\tau$ acts on $\OO$ by $e^x \mapsto (e^x)^\tau =  e^{\tau x}$.
Since the derivations $D(a \wedge b)$ are
defined in terms of multiplication in $\OO$, it follows that 
$(D(a \wedge b)c)^{\tau} = D(a^{\tau} \wedge b^{\tau}) c^{\tau}$
for all $a, b, c \in \OO$ and thus, by linearity, 
\begin{equation*}
(D(w) c)^{\tau} = D(w^{\tau}) c^{\tau} \text{\; for all \;} 
w \in \wedge^2 \op{Im}(\OO), \; c \in \OO.
\end{equation*}
Let $\lbrace 0, 1, x, y \rbrace \subseteq \F_8$ be 
the subset corresponding to any line of $\mathbb{P}^2(\F_2)$ containing $1$.  
Define
\begin{equation*}
\Delta  = e^x \wedge e^y + (e^x \wedge e^y)' + (e^x \wedge e^y)'' \in \wedge^2 \op{Im}(\OO).
\end{equation*}
Note that the element $\pm \Delta$ is independent of choice of the line and choice of the 
ordered pair $(x, y)$, since the Frobenius action permutes the three lines containing $1$, and
interchanging $(x, y)$ changes $\Delta$ by a sign.
To be specific, we choose $(x, y) = (\alpha, \alpha^3)$. Then 
\begin{equation*}
\Delta = e_{1} \wedge e_3 + e_{2} \wedge e_6 + e_4 \wedge e_5.
\end{equation*}
\begin{lemma}
(a) $\op{ker}(D)$ has a basis given by $\Delta, \Delta^{\op{M}} , \dotsb, \Delta^{\op{M}^6} $. 
\par
(b) One has $[e_{13}, e_{26}] = 0$.
\label{l-kerD}
\end{lemma}
\begin{proof}
(a) Let $w \in \wedge^2 \op{Im}(\OO)$ and $c \in \OO$. 
Since $D(w) c = 0$ implies $D(w^{\op{M}}) c^{\op{M}} = 0$,
it suffices to show that $D(\Delta)= 0$.
Lemma \ref{l-formula-for-D} implies that if $\lbrace 0 , 1, x , y \rbrace \subseteq \F_8$ is 
a subset corresponding to a line in $P^2(\F_2)$, then
$D(e^x \wedge e^y ) e^1 = 0$, since
$x + y = 1$. So $D(\Delta) e^1 = 0$. 
Since $\Delta' = \Delta$, the equation $D(\Delta) e^{x} = 0$ implies
$ 0 = D( \Delta') (e^{x})' = D(\Delta) e^{x^2}$. 
So it suffices to show that $D(\Delta)$ kills $e^{\alpha}$ and $e^{\alpha^3}$.
This is an easy calculation using Lemma \ref{l-formula-for-D}.
This proves that $\Delta, \Delta^{\op{M}} , \dotsb, \Delta^{\op{M}^6} \in \op{ker}(D)$.
One verifies that these seven elements are linearly independent.
\par
(b)  Write $X = [e_{1 3}, e_{2 6}]$.
From part (a), we know that $e_{13} + e_{2 6} + e_{4 5} = 0$.
It follows that $[e_{1 3}, e_{2 6} ] = [e_{2 6}, e_{45}]$, that is, $X$ is Frobenius invariant. 
So it suffices to show that $X$ kills $e_0, e_1, e_3$.
The equation $X e_0 = 0$ is immediate. Verifying
$X e_1 = 0$ is an easy calculation using Lemma \ref{l-formula-for-D}.
The calculation for $e_3$ is identical to the calculation for $e_1$ since
 $e_{1 3} = -e_{3 1}$ and $e_{2 6} = -e_{6 2}$.
\end{proof}
\begin{definition}[The roots and coroots] 
Lemma \ref{l-kerD} implies that $(\C e_{1 3} + \C e_{2 6})$
 is an abelian subalgebra of $\mathfrak{g}$.
We fix this Cartan and call it $H$.
Fix a pair of coroots 
\begin{equation*}
H_{\pm \beta} = \pm H_{\beta} = \mp i e_{1 3}
\end{equation*}
in $H$.
Using Frobenius action, we obtain the six coroots
$ \pm  \lbrace H_{\beta}, H_{\beta}', H_{\beta}'' \rbrace$
corresponding to the short roots. 
The six coroots corresponding to 
the long roots are $ \pm \lbrace H_{\gamma} , H_{ \gamma}', H_{ \gamma}'' \rbrace$
where
\begin{equation*}
H_{ \pm \gamma} = \pm H_{\gamma} =  \pm \tfrac{1}{3} i (e_{13} -  e_{26}).
\end{equation*}
We shall define a basis of $\mathfrak{g}$ containing $H_{\beta}$ and 
$H_{\gamma}$.
The scaling factors like $\tfrac{1}{3}$ are chosen to make sure
that the structure constants of $\mathfrak{g}$ with respect to this basis are 
integers and are smallest possible.
Define roots $\beta, \gamma$ such that 
\begin{equation*}
\Bigl( \begin{smallmatrix} \beta(H_{\beta}) & \beta(H_{\gamma}) \\ \gamma(H_{\beta}) & \gamma(H_{\gamma}) \end{smallmatrix} \Bigr) 
= \bigl( \begin{smallmatrix} 2 & -1 \\ -3 & 2 \end{smallmatrix} \bigr)
\end{equation*}
is the Cartan matrix of $\mathfrak{g}_2$.
So $\lbrace \beta, \gamma \rbrace$ is a pair of simple roots with $\beta$ being the short root;
see Figure \ref{fig-roots-and-coroots}.
Let $\Phi_{\text{short}}$ be the set of six short roots and let $\Phi$ be the set of twelve roots of $\fg$.
Note that the Frobenius acts on $H$ by anti-clockwise rotation of $120$--degrees and complex
conjugation acts by negation. 
\par
Once the roots and coroots have been fixed, the weight space decompositions 
of the two smallest irreducible representations of $\mathfrak{g}$
can be found by simultaneously diagonalizing the actions of
$H_{\beta}$ and $H_{\gamma}$. These weight spaces are described below.
\end{definition}
\begin{figure}
\centerline{
 \begin{tikzpicture}[scale=.6]
  \begin{scope}[xshift=-6cm, scale = .5]
 \draw [arrows = {-Stealth[inset=2pt, length=7pt, open, round]}] (0,0) -- ( 4, 0); 
 \node [right] at (4,0) {\tiny $\beta $};
 \draw [arrows = {-Stealth[inset=2pt, length=7pt, open, round]}] (0,0) -- ( 2, 3.464);  
 \draw [arrows = {-Stealth[inset=2pt, length=7pt, open, round]}] (0,0) -- (-2, 3.464);
 \node [above left] at (-2,3.464) {\tiny $\beta' $};
 \draw [arrows = {-Stealth[inset=2pt, length=7pt, open, round]}] (0,0) -- (-4, 0);
  \draw [arrows = {-Stealth[inset=2pt, length=7pt, open, round]}] (0,0) -- (-2, -3.464);  
  \node [below left] at (-2,-3.464) {\tiny $\beta''$};
 \draw [arrows = {-Stealth[inset=2pt, length=7pt, open, round]}] (0,0) -- ( 2,-3.464);  
 \draw [arrows = {-Stealth[inset=2pt, length=7pt, open, round]}] (0,0) -- ( 6, 3.464);
    \node [above right] at (6,3.464) {\tiny $\gamma''$};    
 \draw [arrows = {-Stealth[inset=2pt, length=7pt, open, round]}] (0,0) -- ( 0, 6.928);  
 \draw [arrows = {-Stealth[inset=2pt, length=7pt, open, round]}] (0,0) -- (-6, 3.464);
  \node [above left] at (-6,3.464) {\tiny $\gamma$};   
 \draw [arrows = {-Stealth[inset=2pt, length=7pt, open, round]}] (0,0) -- (-6,-3.464);  
 \draw [arrows = {-Stealth[inset=2pt, length=7pt, open, round]}] (0,0) -- ( 0, -6.928);  
   \node [below] at (0,-6.928) {\tiny $\gamma'$};   
 \draw [arrows = {-Stealth[inset=2pt, length=7pt, open, round]}] (0,0) -- ( 6,-3.464); 
 \end{scope}
  \begin{scope}[xshift=6cm, scale = .5]
 \draw [arrows = {-Stealth[inset=2pt, length=7pt, open, round]}] (0,0) -- ( 4, 0); 
 \node [right] at (4,0) {\tiny $H_{\beta} =  -i e_{13} $};
 \draw [arrows = {-Stealth[inset=2pt, length=7pt, open, round]}] (0,0) -- ( 2, 3.464);  
 \draw [arrows = {-Stealth[inset=2pt, length=7pt, open, round]}] (0,0) -- (-2, 3.464);
 \node [above] at (-2.4,3.464) {\tiny $-i e_{26}  $};   
 \draw [arrows = {-Stealth[inset=2pt, length=7pt, open, round]}] (0,0) -- (-4, 0);
  \draw [arrows = {-Stealth[inset=2pt, length=7pt, open, round]}] (0,0) -- (-2, -3.464);  
  \node [below] at (-2.4,-3.464) {\tiny $-i e_{45}  $};   
 \draw [arrows = {-Stealth[inset=2pt, length=7pt, open, round]}] (0,0) -- ( 2,-3.464);  
 \draw [arrows = {-Stealth[inset=2pt, length=7pt, open, round]}] (0,0) -- ( 6, 3.464); 
 \draw [arrows = {-Stealth[inset=2pt, length=7pt, open, round]}] (0,0) -- ( 0, 6.928);  
 \draw [arrows = {-Stealth[inset=2pt, length=7pt, open, round]}] (0,0) -- (-6, 3.464);
  \node [above] at (-6.6,3.464) {\tiny $3 H_{\gamma}  $};   
 \draw [arrows = {-Stealth[inset=2pt, length=7pt, open, round]}] (0,0) -- (-6,-3.464);  
 \draw [arrows = {-Stealth[inset=2pt, length=7pt, open, round]}] (0,0) -- ( 0, -6.928);  
 \draw [arrows = {-Stealth[inset=2pt, length=7pt, open, round]}] (0,0) -- ( 6,-3.464); 
 \end{scope}
 \end{tikzpicture}
 }
 \caption{The roots on the left and the coroots on the right.}
 \label{fig-roots-and-coroots}
 \end{figure}
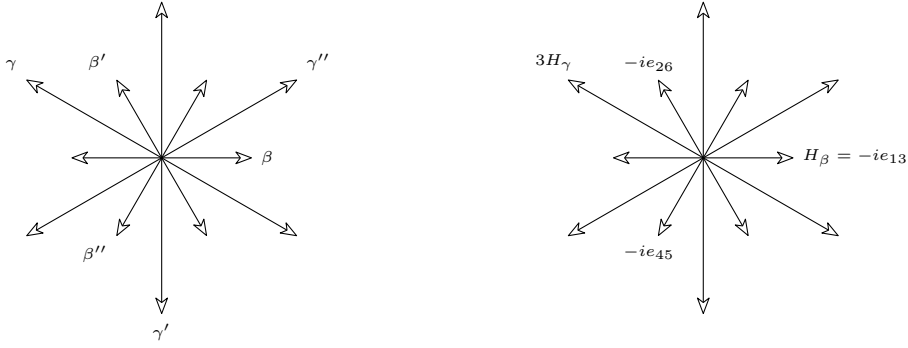
\begin{topic}{\bf The standard representation:}
Write $V = \op{Im}(\OO) \otimes \C$. This is the standard representation of $\fg$. 
Define the vectors $ v_0, v_{\pm \beta},  v'_{\pm \beta}, v''_{\pm \beta}$ in $V$  by choosing 
\begin{equation*}
v_0 = e_0 \text{\; and \;} v_{\pm \beta} = (\pm i e_1 + e_3).
\end{equation*}
One easily verifies that $v_{\psi}$ spans the weight space $V_{\psi}$ for each short root $\psi$.
See Figure \ref{fig-standard-rep}. One has the weight space decomposition:
$V = \C v_0 \oplus \bigl( \oplus_{\psi \in \Phi_{\text{short}}} \C v_{\psi} \bigr)$.
\end{topic}
\begin{topic}{\bf The adjoint representation: }
If $\psi$ is a short root of $\mathfrak{g}$, define
\begin{equation*}
E_{\psi} = \tfrac{1}{2}D(v_0 \wedge v_{\psi}).
\end{equation*}
If $\op{\nu}$ is a long root of $\mathfrak{g}$,
then there exists a unique short root $\psi$ such that $\nu = \psi - \psi'$. Define
\begin{equation*}
E_{\nu}  = \tfrac{1}{6} D(v_{\psi} \wedge v'_{-\psi}).
\end{equation*}
One easily verifies that $E_{\rho}$ spans the root space $\fg_{\rho}$ for each root $\rho \in \Phi$.
One has the root space decomposition: 
$\mathfrak{g} = H \oplus \bigl( \oplus_{\rho \in \Phi} \C E_{\rho} )$.
Note that
\begin{equation*}
E_{ \beta} = \tfrac{1}{2} D( v_0 \wedge v_{\beta} ) = \tfrac{1}{2} D( e_0 \wedge (i e_1 + e_3) ) 
=\tfrac{1}{2}(- i e_{1 0} + e_{0 3}),
\end{equation*}
and 
\begin{equation*}
E_{\gamma}  
= \tfrac{1}{6} D( v_{-\beta} \wedge v_{\beta}') 
= \tfrac{1}{6} D( (-i e_1 + e_3) \wedge (i e_2 + e_6))
=\tfrac{1}{6} (e_{12} + e_{36} -i( e_{23} + e_{16})  ).
\end{equation*}
To write down the other $E_{\rho}$'s, apply the $S_3$ symmetry generated by
complex conjugation and Frobenius. 
\end{topic}
\begin{theorem}
The set 
$\lbrace H_{\beta}, H_{\gamma} \rbrace \cup \lbrace E_{\rho} \colon \rho \in \Phi \rbrace$
is a Chevalley basis of $\fg$.
\label{th-Chevalley-basis}
\end{theorem}

\begin{proof}[Remark on proof]
Checking that these generators of $\mathfrak{g}$ obey the commutation rules dictated by the root space
decomposition is a routine verification using their action on the standard representation 
$V$ as described in remark \ref{t-action-on-V}. Because of the visible $S_3$ symmetry of our
construction, only few cases need to be checked. 
\end{proof}
\par
{\bf Warning:} 
Identify $(\wedge^2 \op{Im}(\OO) \otimes \C)$
 with $\mathfrak{so}_7(\C)$ in the standard manner (see \cite{FH:RT}, page 303)
so that $e_i \wedge e_j$ gets identified with the skew symmetric matrix $2(\mathcal{E}_{i j} - \mathcal{E}_{j i})$
where $\mathcal{E}_{i j}$ is the matrix with rows and columns indexed by $\Z/7 \Z$
whose only nonzero entry is $1$ in the $(i, j)$-th slot.
Let $\nu$ be a long root. Write $\nu  = \psi - \psi'$ for a short root $\psi$.
It is curious to note that  
\begin{equation*}
[ D(v_0 \wedge v_{\psi}), D(v_0 \wedge v_{-\psi}') ]_{\mathfrak{g}}
=4 [E_{\psi}, E_{-\psi'}]_{\mathfrak{g}} 
=12 E_{\nu}
=- D [v_0 \wedge v_{\psi},v_0 \wedge v_{-\psi}']_{\mathfrak{so}_7(\C)},
\end{equation*}
even though $-D$ is not a Lie algebra homomorphism.
%
\begin{remark}[Action of the Chevalley basis on the standard representation]
\label{t-action-on-V}
The action of the vectors
$\lbrace E_{\rho} \colon \rho \in \Phi \rbrace$
on the weight vectors
$\lbrace v_0 \rbrace \cup \lbrace v_{\psi} \colon \psi \in \Phi_{\text{short}} \rbrace$
is determined up to scalars by weight consideration
since $[\fg_{\rho}, V_{\rho'}] \subseteq V_{\rho + \rho'}$ and each weight space $V_{\rho}$ is at most
one dimensional. The non-trivial scalars are determined by the following rules:
Let $\psi$ be a short root and let $\rho$ be a root such that $\psi + \rho$ is also a short root. Then 
\begin{equation}
 E_{\psi} v_0 = v_{\psi}, \;\; E_{\psi} v_{-\psi} = -2 v_0, \;\; \text{\; and \;}
E_{\rho} v_{\psi} = \pm v_{\rho + \psi}
\label{eq-action-on-V}
\end{equation}
 where the plus sign holds if and only if 
$v_{\rho + \psi}$ is equal to $v'_{\psi}$ or $-v''_{\psi}$. In other words, the plus sign
holds if and only if the movement from $\psi$ to $(\rho + \psi)$
in the direction of $\rho$ defines an anti-clockwise rotation of angle less than $\pi$ around the origin. 
The relations in equation \eqref{eq-action-on-V} are easily verified using Lemma \ref{l-formula-for-D}.
Only a few relations need to be checked, because of the $S_3$ symmetry.
The nontrivial scalars involved in this action are indicated in Figure \ref{fig-standard-rep}
next to the dashed arrows. For example, the $-2$ next to the horizontal arrow means 
that $E_{\beta} v_{-\beta} = -2 v_0$.
\end{remark}
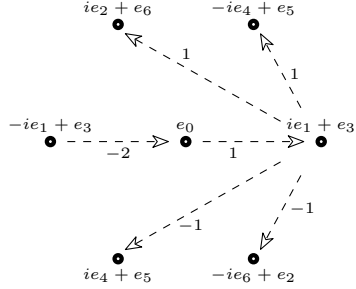
\begin{figure}
\centerline{
 \begin{tikzpicture}[scale=.45]
  \begin{scope}[ scale = 1]
 \draw [ultra thick] (0,        0) circle (3pt); \node [above] at (0,0) {\tiny $e_{0} $};
 \draw [ultra thick] ( 4,0       ) circle (3pt); \node [above] at ( 4, 0       ) {\tiny $ i e_{1} + e_{3} $};
 \draw [ultra thick] ( 2,3.464) circle (3pt); \node [above] at ( 2, 3.464) {\tiny $ -i e_{4} + e_{5} $};
 \draw [ultra thick] (-2,3.464) circle (3pt); \node [above] at (-2, 3.464) {\tiny $ i e_{2} + e_{6} $};
 \draw [ultra thick] (-4,0       ) circle (3pt); \node [above] at (-4, 0       ) {\tiny $ -i e_{1} + e_{3} $}; 
\draw [ultra thick] (-2,-3.464) circle (3pt); \node [below] at (-2,-3.464) {\tiny $ i e_{4} + e_{5} $};
\draw [ultra thick] ( 2,-3.464) circle (3pt); \node [below] at ( 2,-3.464) {\tiny $ -i e_{6} + e_{2} $};
 \draw [dashed, arrows = {-Stealth[inset=2pt, length=7pt, open, round]}] (.5,0) -- (  3.5, 0);  
 \node [below] at (1.4,.1) {\tiny $1$};
  \draw [dashed, arrows = {-Stealth[inset=2pt, length=7pt, open, round]}] (-3.5,0) -- (  -.5, 0);  
  \node [below] at (-2,.1) {\tiny $-2$};
  \draw [dashed, arrows = {-Stealth[inset=2pt, length=7pt, open, round]}] (3.4,1) -- ( 2.2,3.2);  
 \node [ right] at (2.8,2) {\tiny $1$};
  \draw [dashed, arrows = {-Stealth[inset=2pt, length=7pt, open, round]}] (2.8,.6) -- (-1.8, 3.2);  
 \node [ right] at (-.4,2.6) {\tiny $1$}; 
   \draw [dashed, arrows = {-Stealth[inset=2pt, length=7pt, open, round]}] (3.4,-1) -- ( 2.2,-3.2);  
 \node [ right] at (2.8,-2) {\tiny $-1$};
  \draw [dashed, arrows = {-Stealth[inset=2pt, length=7pt, open, round]}] (2.8,-.6) -- (-1.8, -3.2);  
 \node [ right] at (-.5,-2.5) {\tiny $-1$}; 
  \end{scope}
 \end{tikzpicture}
 }
 \caption{Basis for weight spaces in the standard representation $\op{Im}(\OO) \otimes \C$.
 The numbers next to the dashed arrows indicate the scalars involved in action of 
some of the $E_{\psi}$'s as stated in equation \eqref{eq-action-on-V}.
The rest can be worked out from weight consideration
and Weyl group symmetry. }
 \label{fig-standard-rep}
 \end{figure}
\begin{topic}{\bf The irreducible representations of $\mathfrak{g}_2$:}
We finish by describing the finite dimensional irreducible representations of $\mathfrak{g}$
in terms of the standard representation $V$. This was worked out in \cite{HZ:W}.
The description given below follows quickly from the results of \cite{HZ:W}.
\par
Fix the simple roots $\lbrace \beta, \gamma \rbrace$ as in Figure \ref{fig-roots-and-coroots}.
Then the fundamental
weights of $\mathfrak{g}$ are $\mu_1 = -\beta''$ and $\mu_2 = -\gamma'$.
For each non-negative integer $a, b$, let $\Gamma_{a, b}$ denote the finite dimensional
irreducible representation of $\mathfrak{g}$ with highest weight $(a \mu_1 + b \mu_2)$.
The two smallest ones are the standard representation
$V = \Gamma_{1,0}$ and the adjoint representation $\mathfrak{g} = \Gamma_{0,1}$.
\par
Let $\lambda$ be the Young tableau having two rows, corresponding to the partition $(a+b, b)$.
Then $\Gamma_{a, b}$ can be realized as a subspace of the Weyl module $\mathbb{S}_{\lambda}(V)$;
see \cite{HZ:W}, Theorem 5.5.
From \cite{F:YT}, chapter 8, recall 
that the vectors in $\mathbb{S}_{\lambda}(V)$ can be represented in the form 
\begin{equation}
 w =
  \begin{smallmatrix}
w_{1,1} & w_{1,2} & \dotsb & w_{1,b} & \dotsb w_{1,a+b} \\
w_{2,1} & w_{2,2} & \dotsb & w_{2,b} & 
\end{smallmatrix}
\label{eq-w}
\end{equation}
where $w_{i, j} \in V$, modulo the following relations:
\begin{itemize}
\item Interchanging the two entries of a column negates the vector $w$.
\item Interchanging two columns of the same length does not change $w$.
\item For each $1 \leq j \leq b$ and $j <   k \leq a + b$, let $z_1$ (resp. $z_2$)
be the vector obtained from $w$ by interchanging
$w_{1,k}$ with $w_{1,j}$ (resp. $w_{2,j}$). Then $w = z_1 + z_2$.
\end{itemize}
These relations are the {\it exchange conditions}
of \cite{F:YT}, page 81, worked out in our situation. 
\par
The natural surjection from $\otimes^{ a+2 b}V \to \mathbb{S}_{\lambda}(V)$ induces
the $\mathfrak{g}$--action on $ \mathbb{S}_{\lambda}(V)$.
Note that the highest weight of $\Gamma_{a, b}$ is 
$(a \mu_1 + b \mu_2) = (a + b) (-\beta'')  + b \beta'$.  From Figure \ref{fig-standard-rep}, 
recall that  $ v_{-\beta}'' = - i e_4 + e_5$
and $v_{\beta}' = i e_2 + e_6$.
Let $w_{\lambda} \in \mathbb{S}_{\lambda}(V)$ be the vector written in the form
given in equation \eqref{eq-w} whose first row entries are all
equal to $(- i e_4 + e_5)$ and whose second row entries are all equal to $(i e_2 + e_6)$.
Then we find that $w_{\lambda}$ has weight $(a \mu_1 + b \mu_2)$. 
So  $\Gamma_{a, b} = U(\mathfrak{g}) w_{\lambda}$, and
$w_{\lambda}$ is the highest weight vector of $\Gamma_{a, b}$.
\end{topic}
{\bf Acknowledgement:} I would like to thank Jonathan Smith and Jonas Hartwig
for many interesting discussions and helpful suggestions.

\end{document}